\documentclass[12pt]{amsart}
\usepackage[utf8]{inputenc}
\usepackage{times}
\usepackage{amsmath, amssymb, amsthm, amsfonts, verbatim,fullpage}
\usepackage{amscd}
\usepackage{graphicx}
\usepackage[all]{xy}
\usepackage{enumerate}
\usepackage{graphicx} 

\newcommand{\N}{\mathbb{N}}

\newcommand{\PP}{\mathbb{P}} 
\newcommand{\R}{\mathbb{R}}

\newcommand{\Z}{\mathbb{Z}} 
 
\newcommand{\cL}{\mathcal{L}}

\newcommand{\fm}{{\frak m}}

\DeclareMathOperator{\Cr}{Cr}

\DeclareMathOperator{\reg}{reg}

\newtheorem{theorem}{Theorem}[section]
\newtheorem{proposition}[theorem]{Proposition}
\newtheorem{lemma}[theorem]{Lemma}
\newtheorem{corollary}[theorem]{Corollary}
\newtheorem{conjecture}[theorem]{Conjecture}

\theoremstyle{definition}
\newtheorem{remark}[theorem]{Remark}

\numberwithin{equation}{section}

\begin{document}

\title{Interpolation and the weak Lefschetz property}

\author{Uwe Nagel}
\address{Department of Mathematics\\
University of Kentucky\\
715 Patterson Office Tower\\
Lexington, KY 40506-0027 USA}
\email{uwe.nagel@uky.edu}

\author{Bill Trok}
\address{Department of Mathematics\\
University of Kentucky\\
715 Patterson Office Tower\\
Lexington, KY 40506-0027 USA}
\email{william.trok@uky.edu}

\thanks{The first author was partially supported by Simons Foundation grant \#317096.}

\begin{abstract} 
Our starting point is a basic problem in Hermite interpolation  theory, namely determining the least degree of a homogeneous polynomial that vanishes to some specified order at every point of a given finite set. We solve this problem if the number of points is small compared to the dimension of their linear span. This also allows us to establish results on the Hilbert function of ideals generated by powers of linear forms.  The Verlinde formula determines such a Hilbert function in a specific instance. We complement this result and also determine the Castelnuovo-Mumford regularity of the corresponding ideals. 
As applications we establish new instances of conjectures by Chudnovsky and by Demailly on the Waldschmidt constant. Moreover, we show that conjectures on the failure of the weak Lefschetz property by Harbourne, Schenck, and Seceleanu as well as by Migliore, Mir\'o-Roig, and the first author are true asymptotically. The latter also relies on a new result about Eulerian numbers. 
\end{abstract} 

\date{\today}

\maketitle


\section{Introduction} 
    \label{sec:intro}

This work contributes to the theory of Hermite interpolation, the study of Hilbert functions related to the Verlinde formula, the containment problem of comparing ordinary and symbolic powers of an ideal, and the study of Eulerian numbers. It also resolves  conjectures on the presence of the weak Lefschetz property. 

A fundamental problem in the theory of Hermite interpolation  is to determine the least degree of a homogenous polynomial that, given a set of s points $Z = \{P_1,\ldots,P_s\}$ in projective space $\PP^n$ and positive integers $m_1,\ldots,m_s$, vanishes to order $m_i$ at the point $P_i$ for every $i$. This is a very difficult problem. This remains true even if the numbers $m_i$ are all the same, say $k = m_1 = \cdots = m_s$. In this case the polynomials vanishing to order $k$ at every point of $Z$ form an ideal $I_Z^{(k)}$  that is called the \emph{$k$-th symbolic power} of the ideal $I_Z = I_Z^{(1)}$ of $Z$. The above interpolation problem asks for the least degree of a nonzero polynomial in this ideal, denoted $\alpha (I_Z^{(k)})$. We solve this problem if $s = |Z|$ is small compared to $n$ and the points in $Z$ span $\PP^n$. This  is well-known if $|Z| = n+1$  (see Remark~\ref{rem:n+1 points}). However, if $|Z| = n+2$  the answer depends on the position of the points. Specifically, let $t$ be the least integer such that a subset of $t+2$ points of $Z$ is linearly dependent (and so $1 \le t \le n$). We show in Theorem~\ref{thm:n+2 points} that
\[
\alpha (I_Z^{(k)}) = \left \lceil \frac{(2n+2-t) k}{2n-t} \right \rceil. 
\]
If $|Z| = n+3$ and no $n+1$ points of $Z$ are linearly dependent, that is, $Z$ is in \emph{linearly general position}, we determine  $\alpha (I_Z^{(k)})$ in Theorem~\ref{thm:n+3 points}.  The parity of $n$ plays an important role in this case. 

Recall that any point $P$ in $\PP^n$ with coordinates $(a_0 : \ldots : a_n)$ determines, up to a multiple,  a linear form 
$\ell_P = a_0 x_0 + \cdots + a_n x_n$ in the polynomial ring  $R = K[x_0,\ldots,x_n]$ over a field $K$. Combined with Matlis-Macaulay duality, this  allows one to relate  symbolic powers $I_Z^{(k)}$ to ideals generated by powers of the  
linear forms. More precisely, if $|Z| = n+3$, let $\ell_1,\ldots,\ell_{n+3}$ be the linear forms dual to the points of $Z$. 
For every integer $d \ge 1$, the graded $K$-algebra $A = R/(\ell_1^d,\ldots,\ell_{n+3}^d )$ is Artinian,  and it is 
interesting (see, e.g., \cite{AP}) and difficult to determine its Hilbert function $h_A (j) = \dim_K [A]_j$ as $j$ varies. For 
example, Sturmfels and Xi \cite{SX} point out that the celebrated Verlinde formula  gives  $h_A (j)$ for $j = (n+1) 
\frac{d+1}{2}$ where $n$ is odd by assumption if $d$  is even (see Remark~\ref{rem:Verlinde}). We compute the 
Hilbert function of $A$ in a different degree. Namely, we determine the maximum degree $j$ such that $[A]_j \neq 0$. This integer $j$ is called the Castelnuovo-Mumford regularity of $A$, denoted $\reg A$. Moreover, we find the Hilbert function in degree $j = \reg A$ (see Theorem~\ref{thm:n+3 points}). 

The above results have consequences for the so-called containment problem. Given a homogeneous ideal $I$ of 
$R$, the problem is to determine all pairs of integers $(m, k)$ such that  $I^{(k)} \subset I^m$. By now there is an 
extensive literature about this question (see, e.g., \cite{BH, MFO, DHST, GHVT, HaHu, HH}). In order to address the 
containment problem, Bocci and Harbourne \cite{BH} pioneered the use of some asymptotic invariants such as the 
Waldschmidt constant and the resurgence. Determining these invariants for a given ideal is often difficult. Chudnovsky \cite{Chud} and Demailly \cite{Demailly} proposed lower bounds for the Waldschmidt constant of the ideal of a set of points. We establish these conjectures in new instances (see Corollaries~\ref{cor:Chud} and \ref{cor:Dem}). 
Furthermore, we determine the resurgence in new cases (see Corollary~\ref{cor:resurgence}). 

Finally we apply our results to the study of the weak Lefschetz property. A graded Artinian algebra $A$ is said to have the \emph{weak Lefschetz property (WLP)} if it has a linear form $\ell$ such  that multiplication by $\ell$ on $A$  has 
maximal rank from each degree to the next. Deciding if an algebra has the WLP is often a delicate problem,  and 
there is a rich literature on this topic (see, e.g., \cite{Cook, DIV,  HMNW, M, SS, S1, S2}). Of particular interest is the 
case, where $A = R/I$ and the ideal $I$ is generated by powers of general linear forms (see, e.g., \cite{HSS, MMN2, 
MM, MN}). If $A$ is an almost complete intersection, that is, $I$ is generated by $n+2 = 1 + \dim R$ powers,  a 
systematic study of the presence of the WLP was begun in \cite{MMN2}. In particular, if all powers have the same 
degree a complete characterization of the presence of the WLP was proposed. We establish this and a related conjecture in 
\cite{HSS} asymptotically (see Theorem~\ref{thm:asympt wlp failure}). To this end we also derive a result on Eulerian numbers where we utilize a connection to the theory of uniform $B$-splines. 

This note is organized as follows. After presenting some initial results in Section \ref{sec;background}, we consider Hermite interpolation for  sets of $n+2$ points in $\PP^n$ in Section \ref{sec:n+2}. The following section is focused on the case of sets with $n+3$ points. Applications to the containment problem are derived in Section~\ref{sec:containment}. In Section~\ref{sec:wlp} we consider the WLP and also establish the needed result on Eulerian numbers. Open questions motivated by this work are discussed in the final section. 
In particular, we offer a conjecture on properties of differences of Eulerian numbers.


\section{Preparatory Results}  
     \label{sec;background}

Throughout this paper $R$ denotes a polynomial ring 
$K[x_0,\ldots,x_n]$ over an arbitrary field $K$ with its standard grading where $\deg x_i = 1$.  If $I \subset R$ is a homogeneous ideal, then the $K$-algebra $A = \oplus_{j \ge 0} [A]_j$ is  standard graded. Its \emph{Hilbert function} is a map  
$h_A: \Z \to \Z, h_A (j) = \dim_K [A]_j$. The \emph{Castelnuovo-Mumford regularity} is an important invariant of $A$ as it is 
determined by the degrees of the syzygies in a minimal free resolution of $A$ over $R$ or, equivalently, by vanishing of 
cohomology groups. If $A \neq 0$ is artinian it is simply the number (see, e.g., \cite[Corollary 4.4]{Ei})
\[
\reg (A) = \max \{j \; | \; [A]_j \neq 0\} 
\]

We will often use the following duality result. 

\begin{theorem}[\cite{EI}] 
       \label{thm:duality}  
Let   $\wp_1, \dots, \wp_s$ be the ideals of  $s$ distinct points in $\mathbb P^{n}$ that are dual to  linear forms $\ell_1,\dots,\ell_s \in R$. 
Let $(\ell_1^{a_1} ,\dots,\ell_s^{a_s}) \subset R$ be an ideal generated by powers of  $s$ linear forms with positive integers $a_1,\dots,a_n$.   Then one has,  for each integer $j \ge -1 + \max \{a_1,\ldots,a_s\}$,
\[
\dim_K \left [R/ (\ell_1^{a_1} ,\dots,\ell_s^{a_s})  \right ]_j =
\dim_K \left [ \bigcap_{a_i \le j}  \wp_i^{j-a_i +1} \right ]_j .
\]
\end{theorem}

If the points defined by the ideals $\wp_i$ are general  points, then the dimension of the linear system  $[ \wp_1^{b_1} \cap \dots \cap \wp_n^{b_s} ]_j\subset [R]_j $ depends only on the numbers $n, j, b_1,\ldots,b_s$. In order to simplify notation, in this case we denote by 
\[
{\mathcal L}_{n}(j; b_1, b_2,\cdots ,b_s)
\]
the linear system  $[ \wp_1^{b_1} \cap \dots \cap \wp_n^{b_s} ]_j\subset [R]_j $. Note that we view it as a $K$-vector space, not a projective space, when we compute dimensions.
At times we use superscripts to indicate repeated entries. For example,
$\mathcal L_2(j; 5^2, 2^3) = \mathcal L_2 (j; 5, 5, 2, 2, 2)$.

Using Cremona transformations, one can relate two different linear systems. This is often stated only for general  points. We need a more inclusive statement. 

A finite set $Z = \{P_0,\ldots,P_s\}  \subset \PP^n$ of points is said to be in \emph{linearly general position} if any subset of $Z$ with $r +1 \le n+1$ points spans an $r$-dimensional linear subspace. If $Z$ has this property and $s \ge n+1$, then, possibly after a suitable coordinate transformation,  we may assume that $P_0,\ldots,P_n$ are the coordinate points of $\PP^n$. The Cremona transformation (with respect to the coordinate points) is the map
\[
\Cr: \PP^n \setminus V(x_0 \cdots x_n) \to \PP^n, \; 
(a_0 : \ldots : a_n) \mapsto (\frac{1}{a_0} : \ldots : \frac{1}{a_n}). 
\]
Since $Z$ is in linearly general position $\Cr (P_i)$ is defined whenever $i > n$. 

Denoting the ideal of a point $P \in \PP^n$ by $I_P$, one has the following result. 

\begin{lemma}
  \label{lem:Cremona}
Let $Z = \{P_0,\ldots,P_s\}  \subset \PP^n$  be a finite set of  points in linearly general position, where 
$n \ge   2$. Let $j, b_1,\ldots,b_s$ be non-negative integers, with $b_1 \ge \dots \ge b_s$.  Set $t = (n-1) j - (b_1 + \cdots + b_{n+1})$. If $b_i + t \ge 0$ for all $i = 0,\ldots,n$, then
\[
\dim_K \left [\bigcap_{i=0}^s  I_{P_i}^{b_i} \right]_j = \dim_K \left [\bigcap_{i=0}^n  I_{P_i}^{b_i+t} \cap \bigcap_{i= n+1}^s I_{P_i}^{b_i} \right ]_{j+t}.
\]
\end{lemma}

\begin{proof}
The argument for general points (see, e.g., \cite[Theorem 3]{Dumnicky}) works also for points in linearly general position. 
\end{proof}

Using the above notation, one gets the more familiar statement for general points (see \cite{Dumnicky, LU, Nagata}).  

\begin{corollary}
  \label{cor:Cremona}
Let $n \ge   2$ and let $j, b_1,\ldots,b_s$ be non-negative integers, with $b_1 \ge \dots \ge b_s$.  Set $t = (n-1) j - (b_1 + \cdots + b_{n+1})$. If $b_i + t \ge 0$ for all $i = 1,\ldots,n+1$, then
\[
\dim_K \mathcal L_n (j; b_1,\ldots,b_s) = \dim_K \mathcal L_n (j + t; b_1 +t,\ldots,b_{n+1} +t , b_{n+2},\ldots,b_s).
\]
\end{corollary}

Sometimes another simplification is possible. 

\begin{lemma}
  \label{lem:Bezout}
Let $Z = \{P_1,\ldots,P_s\}  \subset \PP^n$  be a finite set of  points in linearly general position, where $n \ge   2$, and    
let $b_1,\ldots,b_s$ be non-negative integers. If $j, b_1,\ldots,b_n$ are positive integers and   $b_1 + \cdots + b_{n} > (n-1) j$, then  
\[
\dim_K \left [\bigcap_{i=1}^s  I_{P_i}^{b_i} \right]_j = \dim_K \left [\bigcap_{i=1}^n  I_{P_i}^{b_i-1} \cap \bigcap_{i= n+1}^s I_{P_i}^{b_i} \right ]_{j-1}.
\]
\end{lemma}

\begin{proof}
This follows as for general points (see \cite[Theorem 4]{Dumnicky}) because the numerical assumption implies that the linear form defining the hyperplane spanned by the first $n$ points divides every form in 
$ \left [\bigcap_{i=1}^s  I_{P_i}^{b_i} \right]_j $. 
\end{proof}

For general sets of points, the last statement takes the following form 
(see \cite{Dumnicky, LU}). 

\begin{corollary}
  \label{cor:Bezout}
Let $n \ge   2$ and let $b_1,\ldots,b_s$ be non-negative integers. If $j, b_1,\ldots,b_n$ are positive integers and   $b_1 + \cdots + b_{n} > (n-1) j$, then  
\[
\dim_K \mathcal L_n (j; b_1,\ldots,b_s) = \dim_K \mathcal L_n (j -1; b_1 -1,\ldots,b_{n} -1 , b_{n+1},\ldots,b_s).
\]
\end{corollary}


\section{Sets of $n+2$ points}  
    \label{sec:n+2}

The goal of this section is to determine the initial degree of any uniform fat point scheme that is supported at $n+2$ points of $\PP^n$ that span $\PP^n$, where $n \ge 2$. We begin by considering the case, where the support is in linearly general position. This is the same as a set of $n+2$ general points. Nevertheless, we prefer the first description as it gives precisely the needed assumption on the support and is also meaningful if the base field $K$ is finite. 

Abusing notation slightly, we say that  a set of linear forms in $R$ is in linearly general position if the set of dual points in $\PP^n$ has this property. 

\begin{lemma}
    \label{lem:n+2 upper reg bound}
Let $\ell_1\ldots,\ell_{n+2} \in R= K[x_0,\ldots,x_n]$ be $n+2$ linear forms in linearly general position. Fix a positive integer $d$ and set $r   =\left \lfloor \frac{(n+2) (d-1)}{2} \right \rfloor$. Then one has 
\[
[R/(\ell_1^d,\ldots,\ell_{n+2}^d )]_{r+1} = 0. 
\]
\end{lemma} 

\begin{proof} We want to show that 
\[
D = \dim_K [R/(\ell_1^d,\ldots,\ell_{n+2}^d )]_{r+1}
\]
is zero. Using Theorem~\ref{thm:duality} we get  
\[
D = \dim_K \cL_n (r+1; (r+2 -d)^{n+2}). 
\]
Now we want to apply Lemma~\ref{lem:Cremona}. We compute 
\[
t = (n-1) (r+1) - (n+1)(r+2-d) = -2r -2 + (n+1)(d-1),  
\]
and so $r+2-d + t = -r-1 + n (d-1)$.  One easily checks that the  latter number is non-negative unless $n = 2$. Hence, if $n \neq 2$ Lemma~\ref{lem:Cremona} is applicable and  gives 
\[
D = \dim_K \cL_n (r-d; r+2-d, ( r+1-2d)^{n+1}). 
\]
The latter linear system is  trivial because no degree $r-d$ form is contained in the $(r-d+2)$-nd power of the ideal of a point. 

It remains to consider the case where $n=2$. Then we have 
\[
D = \dim_K \cL_2 (2d-1; d^4). 
\]
If $d \ge 2$ we may apply Lemma~\ref{lem:Bezout} twice and obtain
\begin{equation*}
\begin{split}
D & = \dim_K \cL_2 (2d-1; d^4) \\
& = \dim_K \cL_2 (2d-2; d^2, (d-1)^2) \\
& = \dim_K \cL_2 (2d-3; (d-1)^4)
\end{split}
\end{equation*}
Repeating this we get 
\[
D = \dim_K \cL_2 (1; 1^4) = \dim_K \cL_2 (0; 1^2) = 0, 
\]
which completes the argument. 
\end{proof}

We also need the following result about monomial complete intersections. 

\begin{lemma}
    \label{lem:Hilb ci}
For an integer $d \ge 2$, consider $B = R/(x_0^d, x_1^d, \ldots,x_n^d)$. It's Castelnuovo-Mumford regularity is $r = (n+1)(d-1)$. Furthermore, the Hilbert function of $B$ is strictly increasing on the closed interval $[0, \lfloor \frac{r}{2} \rfloor ]$. It takes its maximum value precisely at $\frac{r}{2}$ if $r$ is even  and at $\frac{r-1}{2}$, $\frac{r+1}{2}$ if $r$ is odd. 
\end{lemma}

\begin{proof}
This follows from \cite[Theorem 1]{RRR}. 
\end{proof} 

\begin{proposition}
     \label{prop:reg n+2 gen} 
Let $\ell_1\ldots,\ell_{n+2} \in R = K[x_0,\ldots,x_n]$ be $n+2$ linear forms in linearly general position. Then one has for every positive integer $d$ that 
\[
\reg R/(\ell_1^d,\ldots,\ell_{n+2}^d ) =  \left \lfloor \frac{(n+2) (d-1)}{2} \right \rfloor
\]  
\end{proposition} 

\begin{proof} 
Set $r   =\left \lfloor \frac{(n+2) (d-1)}{2} \right \rfloor$. 
Put $B = R/(\ell_1^d,\ldots,\ell_{n+1}^d )$ and consider the multiplication map $\times \ell_{n+2}^d \colon [B]_{r-d} \to [B]_r$. The Hilbert function of $B$ is symmetric, that is, $h_B (j) = h_B ((n+1)(d-1) - j$ for every integer $j$. Hence Lemma~\ref{lem:Hilb ci} shows $h_B (r-d) < h_B (r)$. It follows that $[B/\ell_{n+1}^d B]_r \neq 0$, and so $\reg B/\ell_{n+1}^d B \ge r$. Now equality follows by Lemma~\ref{lem:n+2 upper reg bound} because $B/\ell_{n+1}^d B \cong R/(\ell_1^d,\ldots,\ell_{n+2}^d )$. 
\end{proof}

\begin{remark} 
If $K$ is a field of characteristic zero, there is an alternate argument using  the strong Lefschetz property of the algebra $B$ (see \cite{S1, W} or \cite{RRR}). In fact, combined with Lemma~\ref{lem:Hilb ci} it gives that the map $\times \ell_{n+1}^d \colon [B]_{r-d+1} \to [B]_{r+1}$ is surjective. However, if $K$ has positive characteristic then $B$ does not necessarily have even the weak Lefschetz property  (see \cite{Cook}). In this case, there is some integer $j$ such that   $\times \ell_{n+1} \colon [B]_{j-1} \to [B]_{j}$ fails to have maximal rank. Nevertheless, Lemma~\ref{lem:n+2 upper reg bound} gives that $\times \ell_{n+1}^d \colon [B]_{r-d+1} \to [B]_{r+1}$ is always surjective, regardless of the characteristic of $K$. 
\end{remark}

In order to state a consequence of Proposition~\ref{prop:reg n+2 gen}, we denote the \emph{initial degree} of a homogeneous ideal $I \neq 0$ by 
\[
\alpha (I) = \min \{ j \in \Z \; | \; [I]_j \neq 0\}. 
\]

\begin{proposition}
   \label{prop:initial n+2 gen}
Let $Z \subset \PP^n$   be a subset of $n+2$ points in linearly general position. Then one has for every integer $k > 0$
\[
\alpha (I_Z^{(k)} )= 
\left \lceil \frac{(n+2) k}{n} \right \rceil. 
\]
\end{proposition}

\begin{proof}
Combining Proposition~\ref{prop:reg n+2 gen} and Theorem~\ref{thm:duality}, we get for every integer $j$ that 
\[
0 \neq [R/(\ell_1^d,\ldots,\ell_{n+2}^d)]_j \cong [I_Z^{(j+1-d)}]_j 
\]
if and only if $j \le \left \lfloor \frac{(n+2) (d-1)}{2} \right \rfloor$. Setting $k = j+1-d$ the latter is equivalent to 
$j \le \left \lfloor \frac{(n+2) (j-k)}{2} \right \rfloor$, which is true if and only if $j \ge \frac{(n+2) k}{n}$. Now the claim follows. 
\end{proof}

Before extending this result, let us mention the analogous statements for $n+1$ points. 

\begin{remark}
    \label{rem:n+1 points}
Let $Z \subset \PP^n$ be a subset of $n+1$ points spanning $\PP^n$. Then we may take the variables $x_0,\ldots,x_n$ as linear forms that are dual to the given points. The regularity of a complete intersection is well-known. In particular, we get
\[
\reg R/(x_0^d,\ldots,x_n^d) = (n+1) (d-1). 
\]
Now Theorem~\ref{thm:duality} gives as in the above proof for every $k > 0$ that 
\[
\alpha (I_Z ^{(k)}) = \left \lceil \frac{(n+1) k}{n} \right \rceil. 
\]
\end{remark}

The main result of this section gives the announced extension of the above results. 

\begin{theorem}
    \label{thm:n+2 points}
Let $Z \subset \PP^n$ be a subset of $n+2$ points spanning $\PP^n$. Denote by $\ell_1,\ldots,\ell_{n+2}$ linear forms that are dual to the given points. Let $t$ be the least integer such that a subset of $t+2$ points of $Z$ is linearly dependent, and so $1 \le t \le n$. Then one has for every  $d > 0$ and $k > 0$ that 
\[
\reg R/(\ell_1^d,\ldots,\ell_{n+2}^d ) =  \left \lfloor \frac{(2n+2-t) (d-1)}{2} \right \rfloor \quad \text{and} \quad 
\alpha (I_Z^{(k)}) = \left \lceil \frac{(2n+2-t) k}{2n-t} \right \rceil. 
\]
\end{theorem} 

\begin{proof}
We use induction on $n-t \ge 0$. If $t = n$, then we are done by Propositions~\ref{prop:reg n+2 gen} and \ref{prop:initial n+2 gen}. Let $1 \le t < n$. Then we may assume that the first $t+2$ linearly dependent points span $\PP^t = V(x_{t+1},\ldots,x_n)$, that is, the dual linear forms are $\ell_1,\ldots,\ell_{t+2} \in K[x_0,\ldots,x_t] =S$ and $x_{t+1},\ldots,x_n$. It follows that 
\[
R/(\ell_1^d,\ldots,\ell_{n+2}^d ) \cong S/((\ell_1^d,\ldots,\ell_{t+2}^d)  \otimes_K K[x_{t+1}]/(x_{t+1}^d) \otimes_K \cdots \otimes_K K[x_n]/(x_n^d). 
\]
Applying Proposition~\ref{prop:reg n+2 gen} to the first factor, we conclude that 
\begin{equation*}
\begin{split}
\reg R/(\ell_1^d,\ldots,\ell_{n+2}^d ) & = \reg S/(\ell_1^d,\ldots,\ell_{t+2}^d) + (n-t)(d-1) \\
& =  \left \lfloor \frac{(t+2) (d-1)}{2} \right \rfloor + (n-t)(d-1) \\
& =   \left \lfloor \frac{(2n+2-t) (d-1)}{2} \right \rfloor.
\end{split}
\end{equation*}
Now a computation as in the Proof of Proposition~\ref{prop:initial n+2 gen} gives $\alpha (I_Z^{(k)}) = \left \lceil \frac{(2n+2-t) k}{2n-t} \right \rceil$. 
\end{proof}


\section{Subschemes supported at $n+3$ points}  
\label{sec:n+3 points} 

We now consider a set  $Z \subset \PP^n$ consisting of $n+3$ points or, dually, an ideal generated by powers of  $n+3$ linear forms of $R$. We say that a finite set of linear forms is in linearly general position if the set of points dual to these linear forms is in linearly general position. 

\begin{proposition}
     \label{prop:reg n+3 gen} 
Let $\ell_1\ldots,\ell_{n+3} \in R = K[x_0,\ldots,x_n]$ be $n+3$  linear forms in linearly general position. Then one has for any positive integer $d$ that 
\[
\dim_K [R/(\ell_1^d,\ldots,\ell_{n+3}^d )]_r = 
\begin{cases}
1 & \text{if $n$ is even} \\
\binom{n+\rho}{n} & \text{if $n$ is odd and  $(n+2)$ divides $d-1$ or $d \ge n^2 + n+2$}, 
\end{cases}
\]
where 
\[
r =  
\begin{cases}
\frac{(n+2)(d-1)}{2} & \text{if $n$ is even} \\
\left \lfloor \frac{(n+1) (n+3) (d-1)}{2 (n+2)} \right \rfloor &  \text{if $n$ is odd } 
\end{cases}
\]  
and $\rho = \frac{1}{2} (n+1) (n+3) (d-1) - (n+2) r$ satisfies $0 \le \rho \le n+1$.  
\end{proposition} 

\begin{proof} 
If $d=1$, then $r = 0$ and the claim is clear. Thus, we may assume $d \ge 2$. 
Put 
\[
D = \dim_K [R/(\ell_1^d,\ldots,\ell_{n+3}^d )]_{r}.
\]
Theorem~\ref{thm:duality} says that $D$ is the dimension of a linear system of fat points whose support is in linearly general position and that the dimension is independent of the specific choice of the points. 

Our strategy for the proof is to modify  this linear system by using Lemmas \ref{lem:Cremona} and \ref{lem:Bezout} repeatedly. Note that this is possible since the Cremona transformation of $n+3$ points in linearly general position gives another set of $n+3$ points with this property. 
The arguments will show that the linear 
systems obtained in the process have dimensions independent of the choice of the support. Thus, we may treat the linear systems as those of general points and will denote them by $\cL_n (j; b_1,\ldots,b_{n+3})$ throughout the remainder of the proof.  In particular, we claim 
\[
D = \dim_K \cL_n (r; (r+1 -d)^{n+3}). 
\]

We consider two cases according to the parity of $n$. 
\smallskip  

First assume $n$ is odd and write $n = 2 m -1$. 
Thus, $m \ge 2$ and we get 
\begin{equation}
   \label{eq:r}
r = \left \lfloor \frac{2 m (m+1) (d-1)}{2 m +1} \right \rfloor = m (d-1) + \left \lfloor \frac{ m  (d-1)}{2 m +1} \right \rfloor . 
\end{equation}
By definition, $\rho$ is the remainder of $\frac{1}{2} (n+1) (n+3) (d-1)$ on division by $(n+2)$, which gives $0 \le \rho \le n+1$.

We claim that  Lemma~\ref{lem:Cremona} is applicable $m-1$ times and gives 
\begin{equation} 
\label{eq:first round of Cre}
\begin{split}
D & = \dim_K \cL_{2m-1} (r; (r+1 -d)^{2m+2}) \\
& = \dim_K \cL_{2m-1} (r+t; (r+1 -d)^{2}, (r+1 -d+t)^{2m}) \\
& \hspace*{.5cm} \vdots \\
& = \dim_K \cL_{2m-1} (r + (m-1)t; [r+1 -d + (m-2)t]^{2m-2}, [r+1-d + (m-1)t]^4), 
\end{split}
\end{equation}
where $t = -2 r + 2m (d-1)$. 

In order to check this, grant at first that it is indeed possible to perform $(m-1)$ Cremona transformations. Then we get after $j$ steps 
\[
D =  \dim_K \cL_{2m-1 } [r + j t; (r+1 -d + (j-1)t]^{2j}, [r+1-d + j t]^{2m+2-2j}). 
\]
Hence another Cremona transformation changes the multiplicities by 
\begin{equation*}
\begin{split}
t' & = (2m-2) [r+jt] - 2 j [r+1 -d + (j-1) t] - (2m-2j) [r+1-d+jt] \\
& =  -2 r + 2m (d-1) \\
& = t. 
\end{split}
\end{equation*}
Thus, in each step the multiplicities change by $t$, as claimed. Using Equation~\eqref{eq:r},  it is straightforward to check that 
\[
-(d-1) \le t \le 0. 
\]
It follows that Lemma~\ref{lem:Cremona} can be applied $m-1$ times if 
\begin{equation*}
\begin{split}
0 & \le r+1-d + (m-1)t \\
& = r+1 d + (m-1) [-2r + 2m (d-1)] \\
& = - (2m -3) r + (2m^2 - 2m -1) (d-1). 
\end{split}
\end{equation*}
In fact, it is elementary to check that 
\[
\frac{2 m (m+1) (d-1)}{2 m +1} < \frac{(2 m^2 - 2m - 1) (d-1)}{2m -3}. 
\]
Since the left-hand side is an upper bound for $r$, we conclude that 
\[
c= r+1-d + (m-1)t \ge 1. 
\]
This completes the justification of Equation~\eqref{eq:first round of Cre}. 
Thus,  we have 
\[
D  = \dim_K \cL_{2m-1} (c+d-1; [c -t]^{2m-2}, c^4). 
\]

Now we want to apply Lemma~\ref{lem:Bezout} four times in order to obtain 
\begin{equation} 
\label{eq:first round of Bezout}
\begin{split}
D & = \dim_K \cL_{2m-1} (c+d-1; [c -t]^{2m-2}, c^4) \\
& = \dim_K \cL_{2m-1} (c+d-2; [c -t-1]^{2m-2}, c^3, c-1) \\
& = \dim_K \cL_{2m-1} (c+d-3; [c -t-2]^{2m-2}, c^2, [c-1]^2) \\
& = \dim_K \cL_{2m-1} (c+d-4; [c -t-3]^{2m-2}, c, [c-1]^3) \\
& = \dim_K \cL_{2m-1} (c+d-5; [c -t-4]^{2m-2}, [c-1]^4). 
\end{split}
\end{equation}
As $ c > 0$, this is indeed possible if 
\begin{equation*} 
\label{eq:cond for Bezout}
\begin{split}
0  < b & = (2m-2) [c-t] + c - (2m-2)[c+d-1]  \\
& =  c - (2m-2) )d-1+t) \\
& = r + (2m-1)(d-1) - (m-1)t \\
& = (2m-1) r - (2 m^2 -1) (d-1). 
\end{split}
\end{equation*}
One checks that 
\[
\frac{2 m (m+1) (d-1) - 2m }{2 m +1} > \frac{(2 m^2 - 1) (d-1)}{2m -1} 
\]
if $d \ge 4 m^2 - 2m +2$  or if $(2m+1)$ divides $d-1$. (This is the only place in the proof where we use the assumption on $d$.)
Since the left-hand side is a lower bound for $r$, we get $b > 0$ as desired. Hence applying Lemma~\ref{lem:Bezout} a total of $4 b$ times to the last line of Equation~\eqref{eq:first round of Cre}, we get (see Equation~\eqref{eq:first round of Bezout})
\begin{equation*}
D = \dim_K \cL_{2m-1} (c+d-1 - 4b; [c -t - 4b ]^{2m-2}, [c-b]^4). 
\end{equation*}
This is indeed valid as straightforward computations show that the three integers $c+d-1 - 4b, \ c -t - 4b$, and $c - b$ are all non-negative. 

In the next step we want to apply again Lemma~\ref{lem:Cremona}, this time with 
\[
\tilde{t} = (2m-2) [c + d-1-4b] - (2m-2) [c-t-4b] - 2 [c-b]. 
\]
Comparing with $c - b$, we get $\tilde{t} = -c + b$.
Hence, we obtain 
\begin{equation}
   \label{eq:after last Cremona}
D = \dim_K \cL_{2m-1} (d-1 - 3b; [-t - 3b ]^{2m-2}, [c-b]^2). 
\end{equation}
This is valid if the integers $d-1 - 3b, \ -t - 3b$, and $c - b$ are non-negative. We know that $c - b \ge 0$ and $d-1 \ge -t$. Thus, it suffices to check that $-t \ge 3b$.  This is easily confirmed and concludes the justification of Equation~\eqref{eq:after last Cremona}. Since the support of the linear system is in  linearly general position, Theorem~\ref{thm:duality} yields 
\[
D = \dim_K [R/(x_0^{d+t},\ldots,x_{2m-3}^{d+t}, x_{2m-2}^{d-c-2b}, x_{2m-1}^{d-c-2b})]_{d-1-3b}. 
\]
Put $A = R/(x_0^{d+t},\ldots,x_{2m-3}^{d+t}, x_{2m-2}^{d-c-2b}, x_{2m-1}^{d-c-2b})$. Thus, 
\begin{equation*}
\begin{split}
\reg A & = (2m-2) (d+t-1) + 2 (d-c-2b-1) \\
& = (2m-2) [-2r + (2m+1)(d-1)] - 2 (2m+1) r + (2m^2 + 2m)(d-1) \\
& = -(8m-2) r + (8 m^2 + 2m -2) (d-1). 
\end{split}
\end{equation*}
Using that $d-1-3b = -(6m-3) r + (6m^2 -2) (d-1)$, we get 
\[
\reg A - (d-1-3b) = - (2m+1) r + 2m (m+1) (d-1) = \rho \ge 0. 
\]
Since $A$ is a complete intersection,  duality gives 
\[
D = \dim_K [A]_{d-1-3b} = \dim_K [A]_{\rho}. 
\]
The relations of $A$ have degrees $d+t$ and $d-c-2b$. Now one checks that $d+t = \rho + b$ and $d-c-2b = \rho + 1$. We have seen that $b > 0$ if $d \ge 2 (m+1)$ or if $(2m+1)$ divides $d-1$. 
Thus, our assumptions guarantee 
\[
D =  \dim_K [A]_{\rho} = \dim_K [R]_{\rho} = \binom{2m-1+\rho}{2m-1}, 
\]
as claimed in the case where $n$ is odd. 
\smallskip 

Second, assume $n$ is even. We use the same methods as above. However, this case is far easier. 
In fact, applying Lemma~\ref{lem:Cremona}  to $\cL_n (r; (r+1 -d)^{n+3})$ with $r = \frac{(n+2) (d-1)}{2}$, we get
\begin{equation*}
\begin{split}
D & = \dim \cL_n ({\textstyle \frac{n+2}{2} }(d-1); [{\textstyle \frac{n}{2} }(d-1)]^{n+3}) \\
& = \dim \cL_n ({\textstyle \frac{n}{2}} (d-1); [{\textstyle \frac{n}{2}} (d-1)]^{2}, [{\textstyle\frac{n-2}{2} }(d-1)]^{n+1}) \\
& = \dim_K \cL_{n-2} ({\textstyle \frac{n}{2}} (d-1); [{\textstyle \frac{n-2}{2}} (d-1)]^{n+1}). 
\end{split}
\end{equation*}
Repeating this step a total of $\frac{n}{2}$ times, we obtain 
\[
D = \dim_K \cL_0 (d-1; 0^3) = 1, 
\]
as desired.
\end{proof}

\begin{remark}
  \label{rem:comp st-xi}
In the case, where $n$ is even and $d=2$, the above proposition specializes to \cite[Corollary 7.3]{SX}. 
\end{remark}

\begin{corollary}
    \label{cor:initial deg estimate n+3} 
Let $Z \subset \PP^n$ be a subset of $n+3$ points in linearly general position. Then one has for every integer $k > 0$
\[
\alpha (I_Z^{(k)}) \le \begin{cases}
\left \lceil \frac{(n+2) k}{n} \right \rceil & \text{ if $n$ is even} \\[1em]
\left \lceil \frac{(n+1) (n+3) k}{n^2 + 2 n -1} \right \rceil & \text{ if $n$ is odd and $\frac{1}{2} (n^2 + 2n -1)$ divides $k$ or $k \ge \frac{(n^2 + n + 1) (n^2 + 2n -1)}{2 (n+2)}$.} 
\end{cases} 
\]
\end{corollary} 

\begin{proof}
As above, we combine Proposition~\ref{prop:reg n+3 gen} and Theorem~\ref{thm:duality}. The computation is straightforward if $n$ is even. We provide some details in the more complicated case where $n$ is odd. Then Proposition~\ref{prop:reg n+3 gen} yields $0 \neq [I_Z^{(j+1-d)}]_j$  if 
\[
j \le \left \lfloor \frac{(n+1) (n+3) (d-1)}{2 (n+2)} \right \rfloor \; \text{ and \ $n+2$ divides $d-1$ or }  d \ge n^2 + n +2.  
\]
Setting $k = j+1-d$, the condition on $j$ is equivalent to $j \ge \left \lceil \frac{(n+1) (n+3) k}{n^2 + 2 n -1} \right \rceil$, 
and with this assumption  $d \ge n^2 + n +2$ is equivalent to $k \ge \frac{(n^2 + n + 1) (n^2 + 2n -1)}{2 (n+2)}$. 
Finally, if  $d-1 = t (n+2)$ for some integer $t$, then we get $0 \neq [I_Z^{(k)}]_j$ for $k = \frac{1}{2}(n^2 + 2n -1) t$ 
and $j = \frac{1}{2} (n+1)(n+3) t$. This implies the claim if $k$ is divisible by $\frac{1}{2} (n^2 + 2n -1)$. 
\end{proof}

The main result of this section refines the above statements. 

\begin{theorem}
    \label{thm:n+3 points}
Let $Z \subset \PP^n$ be a subset of $n+3$ in linearly general position. Denote by $\ell_1,\ldots,\ell_{n+3}$ linear forms that are dual to the given points.  Then one has 
\[
\reg R/(\ell_1^d,\ldots,\ell_{n+3}^d ) = \begin{cases}
\frac{(n+2)(d-1)}{2} & \text{if $n$ is even and $d > 0$} \\
\left \lfloor \frac{(n+1) (n+3) (d-1)}{2 (n+2)} \right \rfloor &  \text{if $n$ is odd and } \\
&   \text{ $(n+2)$ divides $d-1$ or $d \ge n^2 + n+2$} 
\end{cases}
\]
and 
\[
\alpha (I_Z^{(k)}) = \begin{cases}
\left \lceil \frac{(n+2) k}{n} \right \rceil & \text{ if $n$ is even and $k > 0$} \\[1em]
\left \lceil \frac{(n+1) (n+3) k}{n^2 + 2 n -1} \right \rceil & \text{ if $n$ is odd and $\frac{1}{2} (n^2 + 2n -1)$ divides $k$ or $k \ge \frac{(n^2 + n + 1) (n^2 + 2n -1)}{2 (n+2)}$.} 
\end{cases} 
\]
Moreover, if $n$ is odd,  then $\alpha (I_Z^{(k)}) \ge \left \lceil \frac{(n+1) (n+3) k}{n^2 + 2 n -1} \right \rceil $ for every $k > 0$ and $\reg R/(\ell_1^d,\ldots,\ell_{n+2}^d ) \le \left \lfloor \frac{(n+1) (n+3) (d-1)}{2 (n+2)} \right \rfloor$ for every $d > 0$. 
\end{theorem} 

\begin{proof}
If $I$ is any homogeneous ideal in a polynomial ring of more than one variable, then $ 1 =\dim_K [I]_j $ for some $j$ implies $\alpha (I) = j$. Thus, Proposition~\ref{prop:reg n+3 gen} and 
Theorem~\ref{thm:duality} yield the claim if $n$ is even and if $n$ is odd and 
$\frac{1}{2} (n^2 + 2n -1)$ divides $k$. Indeed, in the latter case we saw in the proof of 
Corollary~\ref{cor:initial deg estimate n+3} that the stated divisibility condition is equivalent to the fact 
that $n+2$ divides $d-1$, which implies $\rho = 0$ in Proposition~\ref{prop:reg n+3 gen}. 

It remains to consider the other powers $k$ if $n$ is odd. Set $p = \frac{1}{2} (n^2 + 2n -1)$. We just showed that, for every integer $t > 0$, 
\[
\alpha (I_Z^{(t p)}) = {\textstyle \frac{1}{2}} (n+1)(n+3) t. 
\]
Hence, $I_Z^{(t p)} \supset (I_Z^{(t)})^p$ implies ${\textstyle \frac{1}{2}} (n+1)(n+3) t \le p \cdot  \alpha (I_Z^{(t)})$, which gives 
\[
\frac{(n+1)(n+3) t}{n^2 + 2n -1}  = \frac{(n+1)(n+3) t}{2p} \le \alpha (I_Z^{(t)}). 
\]
Together with Corollary~\ref{cor:initial deg estimate n+3}, the remaining claims for the initial degree 
now follow. A computation using Theorem~\ref{thm:duality} yields the claims for the regularity. 
\end{proof}

We suspect that the equalities in Theorem~\ref{thm:n+3 points} are always true, that is, also for small  $d$ and $k$ if $n$ is odd.  

\begin{remark}
  \label{rem:Verlinde}
In \cite[Theorem 7.2]{SX} Sturmfels and Xi established the following interpretation of the celebrated Verlinde formula  that is at the interface of algebraic geometry and mathematical physics. If $\ell_1,\ldots,\ell$ are $n+3$ linear forms in $R = K[x_0,\ldots,x_n]$ in linearly general position, then 
\[
\dim_K [R/(\ell_1^{2j+1},\ldots,\ell_{n+3}^{2 j + 1})]_{(n+1) j} = 
\frac{1}{2j+1} \sum_{k=0}^{2j} (-1)^{(n+1)k} \left ( \sin {\textstyle \frac{2k+1}{4j+2}} \pi \right )^{-n-1} 
\]
whenever $j \in \begin{cases}
\Z & \text{ if $n$ is even}\\
\frac{1}{2} \Z & \text{ if $n$ is odd}.
\end{cases}$ 

Comparing with the regularity of the algebra $A = R/(\ell_1^{2j+1},\ldots,\ell_{n+3}^{2 j + 1})$ on the 
left-hand side (see Theorem~\ref{thm:n+3 points}), it follows that the Verlinde formula determines 
the Hilbert function of $A$ in one degree, and this degree is less than the degree for which we 
determined the Hilbert function of $A$ in Proposition~\ref{prop:reg n+3 gen}. 
\end{remark}


\section{Conjectures by Chudnovsky and Demailly} 
     \label{sec:containment} 

There has been a lot of interest in the
``ideal containment problem'' of determining all positive integer pairs $(k ,m)$ such that
$I^{(k)} \subseteq I^m$ (see, e.g., \cite{BH, HaHu}).  In this section we consider several invariants that were introduced to study this problem. 

The \emph{Waldschmidt constant} (see \cite{Wa}) of a nonzero homogeneous ideal $I$ is the number 
\[
\widehat\alpha(I) = \lim_{k \rightarrow \infty} \frac{\alpha(I^{(k)})}{k}. 
\] 
This limit always exists and satisfies $\frac{\alpha(I^{(k)})}{k} \ge \widehat\alpha(I)$ for every $k \ge 1$. There are rather few cases for which  the Waldschmidt is known (see, e.g., \cite{BH, MFO, DHST, GHMN, GHVT}). Thus, one seeks good lower bounds. Chudnovsky and Demailly put forward conjectural estimates for an ideal of a finite set of points $Z \subset \PP^n$. We will establish new instances of these conjectures. Note that it is harmless to assume that $Z$ spans $\PP^n$. Otherwise $Z$ is contained in a hyperplane and $\widehat\alpha(I_Z) = 1$ because $\alpha (I_Z^{(k)} )= k$ for every $k \ge 1$. 

We begin by determining the Waldschmidt constant. 

\begin{proposition}
    \label{prop:Waldschmidt}
Let $Z \subset \PP^n$ be a set of points spanning $\PP^n$. Then one has: 
\begin{itemize}

\item[(a)] If $|Z|= n+2$ and $t$ is the least integer such that a subset of $t+2$ points of $Z$ is linearly dependent, then $\widehat\alpha(I_Z) = \left  \lceil \frac{2n+2-t}{2n-t} \right \rceil$. 

\item[(b)] If $|Z| = n+3$ and $Z$ is in linearly general position, then 
\[
\widehat\alpha(I_Z) = \begin{cases}
\left \lceil \frac{n+2}{n} \right \rceil & \text{ if $n$ is even} \\[1ex]
 \left  \lceil \frac{(n+1) (n+3)}{n^2 + 2n -1} \right \rceil & \text{ if $n$ is odd.} 
\end{cases}
\]
\end{itemize}
\end{proposition}

\begin{proof}
This follows by Theorems~\ref{thm:n+2 points} and \ref{thm:n+3 points}. 
\end{proof}

\begin{remark}
   \label{rem:Waldschmidt} 
(i) For completeness, we note that  $\widehat\alpha(I_Z) = \left  \lceil \frac{n+1}{n} \right \rceil$ if  $|Z| = n+1$ by  Remark~\ref{rem:n+1 points}. 

(ii) If $Z$ is a set of \emph{general}  points with  $|Z| \in \{n+2, n+3\}$, then the above result was shown in \cite[Proposition B.1.1]{DHST}. 
\end{remark}

For an arbitrary finite set of points, Chudnovsky proposed in \cite{Chud} the following estimate. 

\begin{conjecture}[Chudnovsky]
   \label{conj:Chudnovsky}
Any finite set of points $Z \subset \PP^n$ satisfies $\widehat\alpha(I_Z) \ge \frac{\alpha(I_Z) +n -1}{n}$. 
\end{conjecture}

Chudnovsky \cite{Chud} showed that this is true if $n = 2$. Furthermore, the conjecture is known if $Z$ is a set of general points  by Fouli, Mantero, and Xie \cite{FMX} (see also \cite{DTG} if $|Z| \ge 2^n$). 

We will derive instances of Chudnovsky's conjecture as a consequence of the following result, where $\fm = (x_0,\ldots,x_n)$ denotes the homogeneous maximal ideal of $R$.  

\begin{proposition} 
    \label{prop:suff Chud}
Let $Z \subset \PP^n$ be a set of at most $n+3$ points spanning $\PP^n$. Suppose also that no three points are collinear if $|Z| = n+2$ and that 
 $Z$ is in linearly general position if $|Z| = n+3$. 
Then one has for every $k \ge 1$
\[
I_Z^{(kn)} \subset \fm^{(n-1)k} \cdot I_Z^k. 
\]
\end{proposition}

\begin{proof}
Denote by $e^+ (I_Z)$ the maximum degree of a minimal generator of the ideal $I_Z$. Then, by \cite[Proposition 2.3]{HaHu}, it is enough to show that 
\begin{equation}
    \label{eq:suff Chud}
\alpha (I_Z^{(nk)}) \ge k\, e^+ (I_Z) + k\, (n-1). 
\end{equation}
Let $|Z| = n+1$. Then $e^+ (I_Z) = 2$, and Inequality~\eqref{eq:suff Chud} is an equality by Remark~\ref{rem:n+1 points}. 

Let $|Z|= n+2$. As above, let $t$ be the least integer such that a subset of $t+2$ points of $Z$ is linearly dependent. Our assumption says $t \ge 2$. Hence \cite[Theorem C]{HSV} gives $e^+ (I_Z) = 2$. Now a computation using  Theorem~\ref{thm:n+2 points} shows that Inequality~\eqref{eq:suff Chud} is true. 

Finally, let $|Z| = n+3$. Then one has $e^+ (I_Z) = 2$ if $n \ge 3$ and $e^+ (I_Z) = 3$ if $n=2$ (see, e.g., \cite[Theorem 1.1]{N}. Now Theorem~\eqref{thm:n+3 points} implies the desired inequality. 
\end{proof}

Note that the assertion of Proposition~\eqref{prop:suff Chud} is also known if $Z \subset \PP^n$ is a set of general points with $|Z| \ge 2^n$ (see \cite[Theorem 1]{DTG}). 

\begin{corollary}
     \label{cor:Chud} 
Let $Z \subset \PP^n$ be a set of at most $n+3$ points spanning $\PP^n$. Suppose  that 
 $Z$ is in linearly general position if $|Z| = n+3$. Then Chundnosky's conjecture is true for $Z$. 
\end{corollary}

\begin{proof}
By \cite[Lemma 3.2]{HaHu}, this follows from Proposition~\ref{prop:suff Chud}, unless $|Z| = n+2$ and $Z$ contains three collinear points. In the latter case, the claim is implied by Theorem~\ref{thm:n+2 points} with $t=1$. 
\end{proof}

For an arbitrary set of points $Z \subset \PP^n$ and any $k \ge1$,  
Esnault and 
Viehweg showed in  \cite{EV} that  $\widehat\alpha(I_Z) \ge \frac{\alpha(I_{Z}^{(k)}) +1}{n + k-1}$ if the characteristic of the base field $K$ is zero.  Demailly \cite{Demailly} proposed the following strengthening of Chudnovsky's conjecture. 

\begin{conjecture}[Demailly]
   \label{conj:Demailly}
Any finite set of points $Z \subset \PP^n$ satisfies $ \widehat\alpha(I_Z) \ge \frac{\alpha(I_{Z}^{(k)}) +n -1}{n + k -1}$ for every $k \ge 1$. 
\end{conjecture}

If $k = 1$, this is Conjecture~\ref{conj:Chudnovsky}. Recently, Malara, Szemberg, and Szpond established Demailly's conjecture for sets of general points $Z \subset \PP^n$ with $|Z| \ge (k+1)^n$ (see  \cite{MSS}). Our results imply: 

\begin{corollary}
     \label{cor:Dem} 
Let $Z \subset \PP^n$ be a set of at most $n+3$ points spanning $\PP^n$. If $|Z| = n+3$ suppose  that 
 $Z$ is in linearly general position. Then Demailly's conjecture is true for $Z$ and every $k \ge 1$, unless $|Z| = n+3$ and $n$ is odd. In the latter case,  Demailly's conjecture holds whenever  $k \ge \frac{(n^2 + n + 1) (n^2 + 2n -1)}{2 (n+2)}$. 
\end{corollary}

\begin{proof}
Using Remark~\ref{rem:n+1 points} as well as Theorems~\ref{thm:n+2 points} and \ref{thm:n+3 points}, this follows by straightforward computations. 
\end{proof}

As a final application, we determine further invariants that were introduced to study the containment problem. Following  Bocci and Harbourne \cite{BH},  
the {\em resurgence} of a homogeneous $I$ is 

\[
\rho(I) = \sup\left\{\frac{m}{k} \; | \;  I^{(k)} \not\subseteq I^m \right \}.
\]
Later an asymptotic version  was defined by 
Guardo, Harbourne, and Van Tuyl \cite{GHVT} as 
\[
\rho_a(I) = 
\sup\left\{\frac{m}{k}  \; | \; I^{(kt)} \not\subseteq I^{mt} \, \, \, 
\mbox{for all $t \gg 0$.}\right\}
\]

\begin{corollary}
     \label{cor:resurgence} 
Let $Z \subset \PP^n$ be a set of at most $n+3$ points spanning $\PP^n$. Suppose  that no three points are collinear if $|Z| = n+2$ and that 
 $Z$ is in linearly general position if $|Z| = n+3$. 
 Then one has 
 \[
 \rho_a(I_Z) = \rho (I_Z) = \frac{2}{\widehat\alpha(I_Z)}. 
 \]
 \end{corollary}

\begin{proof}
In the proof of Proposition~\ref{prop:suff Chud} we saw that $e^+ (I_Z) = 2$. According to \cite[Theorem 1.2]{GHVT} 
one has $\frac{\alpha(I_Z)}{\widehat\alpha(I_Z)} \le \rho_a(I_Z) \le \rho(I_Z) \le \frac{e^+ (I_Z)}{\widehat\alpha(I)}$, and  the claim follows. 
\end{proof}


\section{Failure of WLP}
   \label{sec:wlp}

A systematic study of the weak Lefschetz property of artinian almost complete intersections was begun in \cite{MMN1}. Their ideals have one more minimal generator than the number of variables of the polynomial ring. This study was continued for ideals whose generators are powers of general linear forms in \cite{MMN2}. There the following conjecture was proposed  in order to complete this line of investigations. Throughout this section $K$ denotes a field of characteristic zero. 

\begin{conjecture}[{\cite[Conjecture 6.6]{MMN2}}]
      \label{conj:uniform}
Consider the ideal  $I = (x_0^{d},\cdots,x_{n}^d,L^{d})$ of the polynomial ring   $R = K[x_0,\ldots,x_n]$, where 
$L \in R$ is a general linear form. If $n \ge 8$ and $n$ is even, then the ring $R/I$ fails the WLP if and only if $d>1$. 
\end{conjecture}

If confirmed we have the following complete description of the presence of the weak Lefschetz property for almost complete intersections generated by uniform powers of general linear forms.  

\begin{remark}
   \label{rem:complete pic}
We adopt the notation of the above conjecture and list references for the statements that have been established.  We assume $d \ge 1$ (since the case $d = 0$ is not interesting).  
\begin{itemize}
\item If $n \le 2$   then $R/I$ has the weak Lefschetz property for every $d \ge 1$  (see \cite{HMNW, SS, MMN2}). 

\item If $n = 3$   then $R/I$ has the weak Lefschetz property if and only if $d \in \{1, 2\}$  (see \cite{MMN2}).

\item If $n = 4$  then $R/I$ has the weak Lefschetz property if and only if $d \in \{1, 2, 3\}$  (see \cite{MMN2}). 

\item If $n \ge 5$ is odd  then $R/I$ has the weak Lefschetz property if and only if $d = 1$  (see \cite{MMN2}). 

\item  If $n = 6$   then $R/I$  has the weak Lefschetz property if and only if $d \in \{1, 2\}$  (see \cite{DIV} for $d = 3$ and  \cite{MMN2} for $d \neq 3$) .

\item If $n \ge 8$  is even  then $R/I$ has the weak Lefschetz property if and only if $d = 1$ (known for $d=2$ by \cite{M}). 
\end{itemize}

The last item is a restatement of Conjecture~\ref{conj:uniform}. It is open for $d \ge 3$. 
\end{remark}

The goal of this section is to provide further evidence by proving Conjecture~\ref{conj:uniform} in new cases and by reducing it to establishing numerical statements for the most part. The verification of these numerical properties has eluded us for the most part. We hope that highlighting them  will motivate further investigations. 

One of our conditions involves Eulerian numbers. For integers $i, j$ with $0 \le j < i$, the Eulerian number $A (i, j)$ is the number of permutations of sets with $i$ elements that have exactly $j$ ascents. It is explicitly given by 
\begin{equation}
     \label{eq:Eulerian number}
A(i, j) = \sum_{k=0}^{j+1} (-1)^k \binom{i+1}{k} (j+1-k)^i. 
\end{equation}

\begin{proposition}
     \label{prop:num conditions}
Given integers $m \ge 2$ and $q$ with $0 \le q \le 2m$, define a polynomial function $P_{m, q}: \R \to \R$ by 
\[
P_{m, q} (t)= \sum_{k=0}^m (-1)^k \binom{2m+2}{k} \binom{m-1 + \lfloor \frac{mq}{2m+1} \rfloor + (q+1) (m-k) + t [ 2m (m+1) - k (2m+1)]}{2m-1}. 
\]
Then one has: 
\begin{itemize}

\item[(a)] If $P_{m, q}$ is not identical to a positive constant, then Conjecture~\ref{conj:uniform} is true for $n = 2 m$ and every $d \gg 0$ such that $d-1-q$ is divisible by $n+1$. 

\item[(b)] If, for some $q$ with $1 \le q \le 2m$,  $P_{m, q} (t) \le 0$ for every integer $t \ge 0$, then Conjecture~\ref{conj:uniform} is true for $n = 2 m$ and every $d \ge n^2 - n+2$ such that $d-1-q$ is divisible by $n+1$. 

\item[(c)] If $P_{m, 0} (t) \le 0$ for every integer $t \ge 1$, then Conjecture~\ref{conj:uniform} is true for $n = 2 m$ and every $d$ such that $d-1$ is divisible by $n+1$. 

\item[(d)] If $\sum_{k=0}^{m} (-1)^k \binom{2m+2}{k} [2 m (m+1) -k (2m+1)]^{2m-1} \neq 0$ for some integer $m \ge 2$, then 
Conjecture~\ref{conj:uniform} is true for $n = 2 m$ and every $d \gg 0$.

\item[(e)] If $A (2m-1, m-1) - 2 A(2m-1, m-2) + A(2m-1, m-3) < 0$ for some integer $m \ge 5$, then 
Conjecture~\ref{conj:uniform} is true for $n = 2 m$ and every $d \gg 0$.
\end{itemize}

\end{proposition}

\begin{proof}
We proceed in several steps. 

(I) We begin by describing the use of Theorem~\ref{thm:n+3 points} for our argument. 
Adopting the notation of Conjecture~\ref{conj:uniform}, set $A = R/I$, where $I = (x_0^{d},\cdots,x_{n}^d,L^{d})$, $L \in R$ is a general linear form, and $n = 2m$. We consider multiplication by a general linear form $\ell \in R$ on $A$ in one specific degree. Put 
\[
r = \left \lfloor \frac{(n+1) (n+3) (d-1)}{2 (n+2)} \right \rfloor =  \left \lfloor \frac{2 m (m+1) (d-1)}{2 m +1} \right \rfloor 
\]
and consider the exact sequence 
\begin{equation}
    \label{eq:key map}
    [A]_{r-1} \stackrel{\ell}{\longrightarrow} [A]_r \longrightarrow [A/\ell A]_r \longrightarrow 0. 
\end{equation}
Note that $A/\ell A$ is isomorphic to a polynomial ring in $n = 2m$ variables over $K$ modulo an ideal generated by  powers of $n+2$ general linear forms. Hence Theorem~\ref{thm:n+3 points} gives $\reg A/\ell A = r$, and in particular $[A/\ell A]_r \neq 0$, provided $d \ge n^2 - n+2$ or $d-1$ is divisible by $n+1$.  It follows that under these assumptions the multiplication in the above sequence has maximal rank if and only if it is injective. 

(II) Now we relate the numerical conditions given in the statement to the above considerations. To this end consider  
a polynomial ring $S = K[y_0,\ldots,y_{n+1}]$ and let $B = S/(y_0^d,\ldots,y_{n+1}^d)$. Resolving $B$ over $S$ using the Koszul resolution, we get for the Hilbert function of $B$
\[
h_B (j) = \sum_{k=0}^{n+2} (-1)^k \binom{n+2}{k} \binom{n+1 + j - kd}{n+1}, 
\]
where we use the convention that  $\binom{a}{b} = 0$ if $a < b$. 

Let $L' = y_0 + \cdots + y_{n+1}$ and consider $B/L' B$.  There is change of coordinates giving a graded isomorphism $A \cong B/L' B$. Since $B$ has the strong Lefschetz property and $L'$ is a Lefschetz element of $B$ (see \cite{MMN1}) we get for the Hilbert function of $A$ that 
\begin{equation}
     \label{eq:compare hilb}
h_A (j) = \Delta h_B (j), 
\end{equation}
where $\Delta$ is the difference operator defined for any function $h: \Z \to \Z$ by $\Delta h (j) = h (j) - h (j-1)$.  For $k \in \N$, we define $\Delta^{k+1} h$ recursively by $\Delta^{k+1} h = \Delta (\Delta^k h)$. Using that $n = 2m$, this gives 
\begin{equation} 
    \label{eq:second diff}
\Delta h_A (j) = \Delta^2 h_B (j)  = \sum_{k=0}^{2m+2} (-1)^k \binom{2m+2}{k} \binom{2m - 1  + j - kd}{2m - 1}. 
\end{equation}     
Since $r \le (m+1) d$ the summands for $k \ge m+1$ are zero, and  we obtain 
\begin{equation}
    \label{eq:hilb at r}
\Delta h_A (r) = \sum_{k=0}^{m} (-1)^k \binom{2m+2}{k} \binom{2m - 1  + r - kd}{2m - 1}. 
\end{equation}

Now write $d-1 = t (2m + 1) + q$ with integers $t$ and $q$ where $0 \le q \le 2m$. Then a straightforward computation gives 
\[
r = 2m (m+1) t + m q + \left \lfloor \frac{mq}{2m+1} \right \rfloor . 
\]
Using the last two equations, Formula~\ref{eq:hilb at r} becomes
\begin{equation}
    \label{eq:hilb after subs}
\begin{split}    
\Delta h_A (r) & = \sum_{k=0}^{m} (-1)^k \binom{2m+2}{k} \binom{m-1 + \lfloor \frac{mq}{2m+1} \rfloor + (q+1) (m-k) + t [ 2m (m+1) - k (2m+1)]}{2m-1} \\
& = P_{m, q} (t). 
\end{split}
\end{equation}
Therefore, if for some integer $t \ge 0$ we have $P_{m, q} (t) < 0$, then the multiplication in Sequence \eqref{eq:key map} cannot be injective. Combined with Step (I) it follows that the map fails to have maximal rank, and so $A$ fails to have the weak Lefschetz property. If $P_{m, q} (t) = 0$ we conclude analogously because the multiplication map cannot be an isomorphism. 
This proves assertions (b) and (c). 

(III) In order to show (a) we use another key result of  Section~\ref{sec:n+3 points}. If $P_{m, q}$ is identical to a non-positive constant we conclude as in Step (II). 

Suppose $P_{m, q}$ is not a constant polynomial. Then the limit of $P_{m, q} (t)$ as $t$ approaches infinity is either $\infty$ or $- \infty$. In the former case, it follows that ${\displaystyle \lim_{t \to \infty}} \dim_K [A/\ell A]_r = \infty$. However, this is impossible because Proposition~\ref{prop:reg n+3 gen} gives $\dim_K [A/\ell A]_r \le \binom{4m -1}{2m-1}$ independent of $t$. Thus, we must have ${\displaystyle \lim_{t \to \infty}} \dim_K [A/\ell A]_r = - \infty$, which yields $\Delta h_A (r) < 0$ for $t \gg 0$ and so for $d \gg 0$. Now we conclude again as above. 

(IV) Note that the coefficient of $t^{2m-1}$ in $P_{m, q}$ is $\sum_{k=0}^{m} (-1)^k \binom{2m+2}{k} [2 m (m+1) -k (2m+1)]^{2m-1}$. Since it is non-vanishing by assumption, it follows that $P_{m, q}$ is polynomial in $t$ of degree $2m-1$ for every $q$ with $0 \le q \le 2m$. Therefore we obtain (d) as a consequence of (a). 

(V) It remains to show (e). We begin by considering the Hilbert series of $B$. It is 
\[
H_B (z) = \sum_{j \ge 0} h_B (j) z^j = \left (\frac{1 - z^d}{1-z} \right)^{2m+2}. 
\]
It follows that 
$
\sum_{j \ge 0} \Delta^{2m+2} h_B (j) z^j = \left (1 - z^d \right)^{2m+2}$, 
which implies 
\[
 \Delta^{2m+2} h_B (j) = \begin{cases}
 (-1)^i \binom{2m+2}{i} & \text{ if $j = di$ and $ 0 \le i \le 2m+2$} \\
 0 & \text{ otherwise.}
 \end{cases}
\]
Hence, ignoring zero,  $\Delta^{2m+2} h_B$ changes its sign $2m+2$ times on the interval from $0$ to $(2m+2)d$. It follows that $\Delta^{k} h_B$ changes its sign at most $k$ times on the interval from $0$ to $(2m+2)(d-1) + k$. Since $h_B$ is unimodal and  $\Delta^2 h_B$ is symmetric about $(m+1)(d-1) + 1$ we conclude in particular that the sequence  $(\Delta^2 h_B (j))_{j \ge 0}$ is first increasing and then decreasing until it reaches  a local minimum at $j = (m+1)(d-1) + 1$. 

Now we consider $\Delta^2 h_B (m (d-1))$. Using Formula~\eqref{eq:second diff}, we get
\[
\Delta^2 h_B (m (d-1)) =  \sum_{k=0}^{m} (-1)^k \binom{2m+2}{k} \binom{m - 1  + (m- k)d}{2m - 1}. 
\]
Considered as a polynomial function in $d$, its degree is at most $2m-1$ and the coefficient of $d^{2m-1}$ is 
\[
\frac{1}{(2m-1)!} \cdot  \sum_{k=0}^{m} (-1)^k \binom{2m+2}{k} (m-k)^{2m-1}. 
\]
Comparing this with the explicit definition of Eulerian numbers (see Formula~\eqref{eq:Eulerian number}), a straightforward computation gives 
\[
\sum_{k=0}^{m} (-1)^k \binom{2m+2}{k} (m-k)^{2m-1} = A (2m-1, m-1) - 2 A(2m-1, m-2) + A(2m-1, m-3). 
\]
Since the right-hand side is negative by assumption it follows that $\Delta^2 h_B (m (d-1))$ is a polynomial function in $d$ of degree $2m-1$ with negative leading coefficient, which implies 
$\Delta^2 h_B (m (d-1)) < 0$ for every $d \gg 0$. Combined with the above analysis of the function $\Delta^2 h_B$, we see that $\Delta^2 h_B (j) < 0$ whenever $m (d-1) \le j \le (m+1)(d-1) +1$. Now $m (d-1) \le r \le (m+1(d-1)$ implies $0 > \Delta^2 h_B (r) = P_{m, q} (t)$ for every $t \gg 0$ independent of $q$, and thus the claim follows by (a). 
\end{proof}

\begin{remark}

(i) The argument for part (e) actually shows more. It gives  that $\Delta h_A (j) = \Delta^2 h_B (j) < 0$ whenever $m (d-1) \le j \le r$. We also know that $[A/\ell A]_j \neq 0$ for these $j$. Hence, $\times \ell: [A]_{j-1} \to [A]_j$ fails to have maximal rank if $m (d-1) < j \le r$. 

(ii) It is easy to check that the assumptions in Proposition~\ref{prop:num conditions}(a) - (e) are true for particular choices of $m$, $q$, and $t$.  Examples suggest that the sequence of integers considered in Proposition~\ref{prop:num conditions}(d) $(\sum_{k=0}^{m} (-1)^k \binom{2m+2}{k} [2 m (m+1) -k (2m+1)]^{2m-1})_{m \ge2}$ is strictly decreasing,  and so all these numbers are negative. 

(iii)  The assumption in 
Proposition~\ref{prop:num conditions}(e) is \emph{not} true if $m \in \{3, 4\}$. We will see below that it is true if $m \gg 0$. 
\end{remark}

We now verify the numerical assumption in Proposition~\ref{prop:num conditions}(a)  in infinitely many cases, which gives new evidence for Conjecture~\ref{conj:uniform}. 

\begin{theorem}
   \label{thm:wlp, d-2 divisible} 
The algebra $R/(x_0^{d},\cdots,x_{n}^d,L^{d})$ fails to have the weak Lefschetz property if $n \ge 8$ is even and $d-2 \gg 0$ is divisible by $n+1$. 
\end{theorem}

\begin{proof}
We continue to use the notation employed in the proof of Proposition~\ref{prop:num conditions}. Consider the case where $q = 1$. Then $\lfloor \frac{mq}{2m+1} \rfloor = 0$, and so  $d-2 = t (2m+1)$ and $r = 2m (m+1) t + m$. Hence, Equations~\eqref{eq:compare hilb} and \eqref{eq:hilb after subs} show that 
\[
 P_{m, 1} (0) =  \Delta^2 h_B (m)  = \sum_{k=0}^{m} (-1)^k \binom{2m+2}{k} \binom{m-1 + 2 (m-k)}{2m-1},  
\]
where $B = S/(y_0^2,\ldots,y_{2m+1}^2)$. The  Hilbert function of $B$ is $h_B (j) = \binom{2m+2}{j}$. Thus, we obtain 
\begin{align*}
\Delta^2 h_B (m)  & = \binom{2m+2}{m} - 2 \binom{2m+2}{m-1} + \binom{2m+2}{m-2} \\
& = \frac{(2m+2)!}{m! \cdot (m+4)!} \big [ (m+4) (m+3) - 2 m (m+4) + m (m-1) \big ] \\
& = \frac{(2m+2)!}{m! \cdot (m+4)!} \big [12 - 2m \big ]. 
\end{align*}
Hence it follows that the function $P_{m, 1}$ cannot be identical to a positive constant if $m \ge 6$, and we conclude by Proposition~\ref{prop:num conditions}(a). 

If $m \in \{4, 5\}$,  the claim follows by Proposition~\ref{prop:num conditions}(d) because 
computations reveal that its assumption  is satisfied. 
\end{proof}

Now we are going to show that the assumption in Proposition~\ref{prop:num conditions}(e) is satisfied if $m \gg 0$. To this end we use a connection to the theory of uniform $B$-splines (see \cite{S46}). For $i \in \N$, define a piecewise polynomial  function $B_i: \R \to \R$ with support in the closed interval $[0, i]$ recursively by 
\[
B_1 (x) = \begin{cases}
1 & \text{ if } x \in [0, 1] \\
0 & \text{ otherwise} 
\end{cases}
\]
and 
\[
B_i (x) = \int_0^1 B_{i-1} (x-t)  d t. 
\]
if $i \ge 2$. In \cite{S73} Schoenberg showed that uniform $B$-splines are related to Eulerian polynomials. More precisely, one has the following result (see \cite[Theorem 1.1]{WXX} or \cite[Corollary 2.7]{He}). 

\begin{lemma}
    \label{lem:euler vs spline}
If $i \ge 2$ and $j$ is any integer, then 
\[
(i-1)! B_i (j) = A(i-1, j-1). 
\]
\end{lemma}

Suitably normalized $B$-splines as well as their derivatives converge.  

\begin{theorem}
    \label{thm:derivatives of splines converge}
For any $k \in \N_0$, the sequence of $k$-the derivatives $(B_i^{(k)})_{i \in \N}$ converges to the $k$-th derivative of the Gaussian function and 
\[
{\textstyle (\frac{i}{12})}^{\frac{k+1}{2}} B_i^{(k)} \big (\sqrt{ {\textstyle \frac{i}{12}}} x + {\textstyle \frac{i}{2}} \big ) = {\textstyle \frac{1}{\sqrt{2 \pi}} } \cdot {\textstyle \frac{d^k}{d x^k}} \exp(- {\textstyle \frac{x^2}{2}} ) + O ({\textstyle \frac{1}{i} })
\]
if $i > k+2$. 
\end{theorem}

We use this to derive the promised estimate. 

\begin{proposition}
    \label{prop:Euler number peak} 
If $m \gg 0$, then 
\[
A (2m-1, m-1) - 2 A(2m-1, m-2) + A(2m-1, m-3) < 0. 
\]
\end{proposition}

\begin{proof}
By Lemma~\ref{lem:euler vs spline}, the claimed inequality is equivalent to 
\begin{equation}
    \label{eq:spline estimate} 
B_{2 m} (m) - 2 B_{2m} (m-1) + B_{2m} (m-2) < 0. 
\end{equation}
Using that $\frac{d^2}{d x^2} \exp(- {\textstyle \frac{x^2}{2}} ) = (-1 + x^2) \exp(-\frac{x^2}{2})$, Theorem~\ref{thm:derivatives of splines converge} gives
\[
{\textstyle (\frac{m}{6})}^{\frac{3}{2}} B_{2m}^{(2)} (x) = {\textstyle \frac{1}{\sqrt{2 \pi}} } \cdot \big (-1 + {\textstyle \frac{6}{m} } (x - m )^2 \big )  \exp(- {\textstyle \frac{3}{m} {(x-m)^2}} ) + O ({\textstyle \frac{1}{m} }). 
\]
A straightforward computation shows that $-1 + {\textstyle \frac{6}{m} } (x - m )^2 < 0$ if $m > 24$. Hence, we obtain for $x \in [m-2, m]$ that $B_{2m}^{(2)} (x) < 0$  whenever $m \gg 0$. This in turn implies the desired Inequality~\ref{eq:spline estimate},  and thus completes the argument. 
\end{proof}

Combining the above results, we establish Conjecture~\ref{conj:uniform} asymptotically. 

\begin{theorem}
  \label{thm:asympt wlp failure} 
Let $n \gg 0$ be an integer and 
consider the ideal  $I = (x_0^{d},\cdots,x_{n}^d,L^{d})$ of the polynomial ring   $R = K[x_0,\ldots,x_n]$, where 
$L \in R$ is a general linear form. Then the ring $R/I$ fails the WLP for every  $d \gg 0$. 
\end{theorem}
    
\begin{proof}
If $n$ is odd, then this follows from the more precise \cite[Theorem 6.1]{MMN2}. If $n$ is  even, we use Propositions~\ref{prop:num conditions}(e) and \ref{prop:Euler number peak}. 
\end{proof}


\section{Final Comments and Open Problems} 
\label{sec:final comm} 

We conclude by discussing some problems motivated by this work. We begin with a combinatorial 
problem concerning Eulerian numbers. 

It is well-known that the Eulerian numbers satisfy 
\[
A(n, k) = A (n, n - k - 1). 
\]
Moreover, fixing $n$ the sequence $(A (n, k))_{k \in \N_0}$ is unimodal. More precisely, 
if $n = 2m -1$ is odd, then $(A (n, k))_{k \in \N_0}$ is strictly increasing on the interval $[0, m-1]$ 
and strictly decreasing on $[m-1, 2m-1]$. Thus, the maximum value of the sequence is attained 
precisely if $k = m-1$. If $n = 2m$ is even, then the sequence is strictly increasing on the interval 
$[0, m-1]$, strictly decreasing on $[m, 2m]$, and has exactly two peaks, namely at $k = m-1$ and at  
$k = m$ as $A(2m, m-1) = A (2m, m)$. 

It is natural to consider the behavior of the  sequence of differences
\[
D(n, k) = A(n, k) - A(n, k-1). 
\]
Then the above results say that $D(n, k) > 0$ if and only if $k \in [0, \lfloor \frac{n-1}{2} \rfloor ]$. Based on computations of $D(n, k)$ for many integers $n, k$, we propose the following conjecture. 

\begin{conjecture} 
   \label{conj:Euler numbers} 
Fixing $n \ge 3$, the sequence $(D (n, k))_{k \in \N_0}$ is strictly increasing on the interval $[0, \lfloor \frac{n-1}{2} \rfloor ]$ if $n$ is even or if $n \in \{3, 5, 7\}$. 

If $n = 2m -1 \ge 9$ is odd, then the sequence is strictly increasing on $[1, m-2]$, but one has $D(2m-1, m-2) > D(2m-1, m-1) > 0$. 
\end{conjecture} 

Notice that 
\[
D(2m-1, m-1) - D(2m-1, m-2) = A (2m-1, m-1) - 2 A(2m-1, m-2) + A(2m-1, m-3). 
\]
Hence,  Proposition~\ref{prop:Euler number peak} provides evidence for the above conjecture as it shows $D(2m-1, m-2) > D(2m-1, m-1)$ if $m \gg 0$. Adapting the method of its proof, one obtains further asymptotic results confirming parts of Conjecture~\ref{conj:Euler numbers}. However, new arguments are needed to establish the conjecture entirely. 
We feel that it would be preferable to have combinatorial arguments. In particular, we would like to see a combinatorial proof of Proposition~\ref{prop:Euler number peak}. \smallskip 

As mentioned above, we expect that Theorem~\ref{thm:n+3 points} has an unnecessary  assumption. 

\begin{conjecture}
    \label{conj:unnecessary assump}
Let $Z \subset \PP^n$ be a subset of $n+3$ in linearly general position, where $n$ is odd.  Then one has 
\[
\alpha (I_Z^{(k)}) = 
\left \lceil \frac{(n+1) (n+3) k}{n^2 + 2 n -1} \right \rceil. 
\]
\end{conjecture}

Recall that in Theorem~\ref{thm:n+3 points} we showed that this is true if $\frac{1}{2} (n^2 + 2n -1)$ divides $k$ or $k \ge \frac{(n^2 + n + 1) (n^2 + 2n -1)}{2 (n+2)}$. In fact the method of proof seems to work for every choice of $(n, k)$. However, if the mentioned numerical assumptions are not satisfied the use of Lemmas~\ref{lem:Cremona} and \ref{lem:Bezout} becomes less uniform and seems to require the consideration of many cases. 

Note that establishing Conjecture~\ref{conj:unnecessary assump} removes one of the obstacles for improving Theorem~\ref{thm:asympt wlp failure} and fully showing Conjecture~\ref{conj:uniform} completely. In fact, combined with  Proposition~\ref{prop:num conditions}(a)  the latter conjecure would follow by also proving that the polynomial in $t$ 
\[
P_{m, q} (t)= \sum_{k=0}^m (-1)^k \binom{2m+2}{k} \binom{m-1 + \lfloor \frac{mq}{2m+1} \rfloor + (q+1) (m-k) + t [ 2m (m+1) - k (2m+1)]}{2m-1}  
\]
is not identical to a positive constant whenever $m \ge2$ and $0 \le q \le 2m$.




\begin{thebibliography}{99}


\bibitem{AP}
F. Ardila and A. Postnikov, \emph{Combinatorics and geometry of power ideals}, Trans. Amer. Math. Soc. {\bf 362}  (2010), 4357--4384.

\bibitem{BH}
C.\ Bocci,  B.\ Harbourne,
Comparing powers and symbolic powers of ideals.
J. Algebraic Geom. {\bf 19} (2010),  399--417.

\bibitem{MFO}  
C.\ Bocci,  S.\ Cooper, E.\ Guardo, B.\ Harbourne,  M.\ Janssen, U.\ Nagel,  A.\ Seceleanu,  A.\ Van Tuyl,  T.\ Vu, {\em The Waldschmidt constant for squarefree monomial ideals}, J.\ Algebraic Comb.\ {\bf 44} (2016),  875--904. 


%

\bibitem{Chud}
G.~V.~Chudnovsky,
\emph{Singular points on complex hypersurfaces and multidimensional Schwarz lemma}, In: 
Seminar on Number Theory, Paris 1979--80, pp. 29--69,
Progress in Mathematics {\bf 12}, Birkh\"auser, Boston, Mass., 1981.

\bibitem{Cook}
    D.\ Cook~II,
    \emph{The Lefschetz properties of monomial complete intersections in positive characteristic},
    J.\ Algebra \textbf{369} (2012), 42--58. 
    
\bibitem{Demailly}
J.-P.~Demailly, \emph{Formules de Jensen en plusieurs variables et applications arithm\'etiques}, Bull.\ Soc.\ Math.\ France {\bf 110} (1982), 75--102.    



\bibitem{DIV} 
R.\ Di Gennaro, G.\ Ilardi and J.\ Vall\`es,  
{\it Singular hypersurfaces characterizing the Lefschetz properties}, 
J.\ London Math.\ Soc. (2) {\bf 89} (2014), 194--212.

\bibitem{Dumnicky} M.\ Dumnicki, {\em An algorithm to bound the regularity
and nonemptiness of linear systems in $\PP^n$},   J.\ Symbolic
Comput.\ {\bf 44}  (2009),  1448--1462.

\bibitem{DHST}
M.\ Dumnicki, B. Harbourne, T.\ Szemberg, and H.\ Tutaj-Gasi\'nska, \emph{Linear subspaces, symbolic powers and Nagata type conjectures}, Adv.\ Math.\ {\bf 252} (2014), 471--491. 

\bibitem{DTG}
M.~Dumnicki and H.~Tutaj-Gasi\'nska,  \emph{A containment result in $\PP^n$ and the Chudnovsky conjecture}, Proc.\ Amer.\ Math.\ Soc.\ Proc.\ Amer.\ {\bf 145} (2017), 3689--3694.  

\bibitem{ELS}
L.~Ein, R.~Lazarsfeld, and K.~Smith, 
\emph{Uniform Behavior of Symbolic Powers of Ideals}, 
Invent.\ Math.\ {\bf 144} (2001),  241--252.

\bibitem{Ei} 
D.~Eisenbud, \emph{The geometry of syzygies}, Graduate Texts in Mathematics {\bf 229}, Springer, New York, 2005. 

\bibitem{EI} 
J.\ Emsalem and A.\ Iarrobino, 
{\it Inverse system of a symbolic power $I$}, 
J.\ Algebra {\bf 174} (1995), 1080--1090.

\bibitem{EV}
H.~Esnault and E.~Viehweg.
\emph{Sur une minoration du degr\'e d'hypersurfaces
s'annulant en certains points},  Math.\ Ann.\ {\bf 263} (1983),  75--86.


\bibitem{FMX}
L.~Fouli, P.~Mantero, and Y.~Xie, \emph{Chudnovsky's Conjecture for very general points in 
$\PP^N_k$}, J. Algebra {\bf 498} (2018), 211--227.

\bibitem{GHMN}
A.V.\ Geramita,   B.\ Harbourne, J.\ Migliore, and U.\ Nagel, {\em Matroid configurations and symbolic powers of their ideals}, Trans.\ Amer.\
Math.\ Soc.\  {\bf 369} (2017), 7049--7066. 

\bibitem{GHVT}
E.\ Guardo, B.\ Harbourne, A.\ Van Tuyl,
Asymptotic resurgences for ideals of positive
dimensional subschemes of projective space.
Adv. Math. {\bf 246} (2013) 114--127.



\bibitem{HaHu}
B.\ Harbourne and C.\ Huneke, {\em Are symbolic powers highly evolved?}, Ramanujan Math. Soc.\ {\bf 28A} (2013), 247--266.

\bibitem{HSS}  B.\ Harbourne, H.\ Schenck, and A.\ Seceleanu, {\em Inverse systems, Gelfand-Tsetlin patterns and the weak Lefschetz property}, J. Lond. Math. Soc. (2) {\bf 84} (2011),  712--730.

\bibitem{HMNW} T.\ Harima J.\ Migliore, U.\ Nagel, and J.\ Watanabe:  {\em The Weak and Strong Lefschetz properties for Artinian $K$-Algebras}, J.\ Algebra {\bf 262} (2003), 99-126.

\bibitem{HH}
M.~Hochster and C.~Huneke, 
\emph{Comparison of symbolic and ordinary powers of ideals}, 
Invent.\ Math.\ {\bf 147} (2002),  349--369.

\bibitem{He}
T.\ He, \emph{Eulerian polynomials and B-splines},  J. Comput. Appl. Math. {\bf 236}(2012), 3763--3773.


\bibitem{HSV}
L.\ T.\ Hoa, J.\  St\"uckrad, and W.\  Vogel, {\em Towards a structure theory for projective varieties of degree=codimension $+2$},  J.\ Pure Appl.\ Algebra {\bf 71} (1991),  203--231. 

\bibitem{LU} A.\ Laface and U.\ Ugaglia, {\em On a class of special
linear systems of $\PP^3$}, Trans.\ Amer.\ Math.\ Soc. {\bf 358}
(2006),
 5485--5500.

%


\bibitem{MSS}
G.~Malara, T.~Szemberg, and J.~Szpond, \emph{On a conjecture of Demailly and new bounds on
Waldschmidt constants in $\PP^N$}, J. Number Theory {\bf 189} (2018), 211--219.

\bibitem{M} R. Mir\'o-Roig, {\em Harbourne, Schenck and Seceleanu's conjecture}, J. Algebra {\bf 462} (2016), 54--66. 

\bibitem{MM}
J. Migliore and R. Mir\'o-Roig, \emph{On the strong Lefschetz problem for uniform powers of general linear forms in $k[x,y,z]$},  Proc. Amer. Math. Soc. {\bf 146} (2018), 507--523.

\bibitem{MMN1} J. Migliore, R.M. Mir\'o-Roig and U. Nagel, \emph{Monomial ideals, almost complete intersections and the weak Lefschetz property}, Trans. Amer. Math. Soc. {\bf 363} (2011),  229--257. 

\bibitem{MMN2}  J.\ Migliore, R.M.\ Mir\'o-Roig, and U.\ Nagel, {\em On the weak Lefschetz property for powers of linear forms}, Algebra Number Theory {\bf 6} (2012),  487--526.

\bibitem{MN} 
J. Migliore and U. Nagel, \emph{The Lefschetz question for ideals generated by powers of  linear forms infew variables}, Preprint, 2017. arXiv:1703.07456


\bibitem{Nagata} M.\ Nagata, {\em On the fourteenth problem of
Hilbert}, In:
 Proc.\ Internat.\ Congress Math.\ 1958,  pp. 459--462, Cambridge
Univ.\ Press, New York,  1960.

\bibitem{N}
U.\ Nagel,  {\em On the minimal free resolution of $r+3$ points in
    projective $r$-space}, J.\ Pure Appl.\ Algebra {\bf 96} (1994), 23--38.

\bibitem{RRR}
    L.\ Reid, L.\ G.\ Roberts, M.\ Roitman,
    \emph{On Complete Intersections and Their Hilbert Functions},
    Canad.\ Math Bull.\ \textbf{34} (1991), no.\ 4, 525--535.


\bibitem{SS} H. Schenck and A. Seceleanu, {\em The weak Lefschetz property and powers of linear forms in $k[x,y,z]$}, Proc. Amer. Math. Soc. {\bf 138} (2010), 2335--2339. 

\bibitem{S46}
I. J. Schoenberg, \emph{Contribution to the problem of approximation of equidistant data by analytic functions}, Quart. Appl. Math. {\bf 4} (1946) 45--99,  112--141. 

\bibitem{S73}
I. J. Schoenberg, \emph{Cardinal splines interpolation},  Conference Board of the Mathematical Sciences Regional Conference Series in Applied Mathematics, No. 12. Society for Industrial and Applied Mathematics, Philadelphia, Pa., 1973.

\bibitem{S1} R. Stanley, {\em Weyl groups, the hard Lefschetz theorem, and the Sperner property}, SIAM J. Algebraic Discrete Methods {\bf 1} (1980), 168--184. 

\bibitem{S2}
    R.\ Stanley, \emph{The number of faces of a simplicial convex polytope},
    Adv.\  Math.\ \textbf{35} (1980), 236--238. 

\bibitem{SX}
B.\ Sturmfels and Z.\ Xi, \emph{Sabgi bases of Cox-Nagata rings}, J.\ Eur.\ Math.\ Soc. {\bf 12} (2010), 429--459.      

\bibitem{Wa}
M.~Waldschmidt,
\emph{Propri\'et\'es arithm\'etiques de fonctions de plusieurs variables. II}, 
In S\'eminaire P.\ Lelong (Analyse), 1975/76,
Lecture Notes Math.\ {\bf 578}, Springer, 1977,  108--135. 

\bibitem{WXX}
R. Wang, Y. Xu and  Z. Xu, \emph{Eulerian numbers: a spline perspective}, J. Math. Anal. Appl. {\bf 370} (2010) 486--490.

\bibitem{W} J. Watanabe, {\em The Dilworth number of Artinian rings and finite posets with rank function}, Commutative Algebra and Combinatorics, Advanced Studies in Pure Math., Vol. 11, North
Holland, Amsterdam (1987), 303--312.

\bibitem{XW} 
Y. Xu, and  R. Wang, \emph{Asymptotic properties of B-splines, Eulerian numbers and cube slicing} J. Comput. Appl. Math. {\bf 236} (2011), 988--995.

\end{thebibliography}
\end{document}